\documentclass[12pt]{amsart}
\usepackage{graphicx}
\usepackage{graphics}
\usepackage{amsmath}
\usepackage{amscd}
\usepackage{latexsym}
\usepackage{subfigure}
\usepackage{caption}
\usepackage{hyperref}

\begin{document}

\textwidth 5.9in
\textheight 7.9in

\evensidemargin .75in
\oddsidemargin .75in

\newtheorem{Thm}{Theorem}
\newtheorem{Lem}[Thm]{Lemma}
\newtheorem{Cor}[Thm]{Corollary}
\newtheorem{Prop}[Thm]{Proposition}
\newtheorem{Rm}{Remark}

\def\a{{\mathbb a}}
\def\C{{\mathbb C}}
\def\A{{\mathbb A}}
\def\B{{\mathbb B}}
\def\D{{\mathbb D}}
\def\E{{\mathbb E}}
\def\R{{\mathbb R}}
\def\P{{\mathbb P}}
\def\S{{\mathbb S}}
\def\Z{{\mathbb Z}}
\def\O{{\mathbb O}}
\def\H{{\mathbb H}}
\def\V{{\mathbb V}}
\def\Q{{\mathbb Q}}
\def\Cn{${\mathcal C}_n$}
\def\CM{\mathcal M}
\def\CG{\mathcal G}
\def\CH{\mathcal H}
\def\CT{\mathcal T}
\def\CF{\mathcal F}
\def\CA{\mathcal A}
\def\CB{\mathcal B}
\def\CD{\mathcal D}
\def\CP{\mathcal P}
\def\CS{\mathcal S}
\def\CZ{\mathcal Z}
\def\CE{\mathcal E}
\def\CL{\mathcal L}
\def\CV{\mathcal V}
\def\CW{\mathcal W}
\def\IC{\mathbb C}
\def\IF{\mathbb F}
\def\IK{\mathcal K}
\def\IL{\mathcal L}
\def\IP{\bf P}
\def\IR{\mathbb R}
\def\IZ{\mathbb Z}

\title{On infinite order corks}
\author{Selman Akbulut}
\thanks{Partially supported by NSF grants DMS 0905917}
\address{Department  of Mathematics, Michigan State University,  MI, 48824}
\email{akbulut@math.msu.edu }
\date{\today}
\begin{abstract} 
We construct an infinite order loose cork.
 \end{abstract}

\date{}
\maketitle

\setcounter{section}{-1}

\vspace{-.2in}

\section{Introduction }

A cork is a pair $(W,f)$, where $W$ is a compact contractible Stein manifold, and $f:\partial W\to \partial W$ is an involution, which extends to a self-homeomorphism of $W$, but does not extend to a self-diffeomorphism of $W$.  It follows that, if $F: W\to W$ is any homeomorphism extending $f$, and $W_{F}$ is the smooth structure pulled back from $W$ by $F$, then $W_{F}$ is an exotic copy of $W$ relative to $\partial W$.  We say $(W,f)$ is a cork of $M$,  if there is an imbedding $g: W\hookrightarrow M$  and cutting $g(W)$ out of $M$ and re-gluing with $g\circ f $  produces an exotic copy $M':= M(f,g)$ of $M$: 
 $$M=W\cup_{g} [M-g(W)]  \mapsto M(f,g)=W\cup _{g\circ f}[M-g(W) ]$$
 
The operation $M\mapsto M'$ is called 
{\it cork-twisting} $M$ along $W$. A first example of a cork appeared in \cite{a1}; and in \cite{m} and \cite{cfhs}, it was proven that, any exotic copy $M'$ of a closed simply connected $4$-manifold $M$ is obtained by cork twisting along a contractible manifold as above. Furthermore  in \cite{am}, it was shown that this contractible manifold  can be taken to be a Stein manifold. Since then, the Stein condition has become a part of the definition of a cork 
(e.g. \cite{ay1}, \cite{am}). Cork twisting can make $4$-manifolds exotic either by varying $g$ or $f$. For example,  it is possible to obtain infinitely many distinct exotic copies of a manifold by twisting along a fixed cork $(W,f)$ while varying its imbedding $g$, i.e. corks can be knotted (\cite{ay3}). So it is a natural to ask whether the same can be accomplished by fixing an imbedding $g=g_{0}$, and varying $f$ instead. For simplicity, let us denote $M(f)= M(f, g_{0})$. Clearly to do this we need to abandon the ``involution'' assumption on $f$, then hope that the infinitely many iterations $f^{n}=f\circ f\circ \cdots \circ f$  induce distinct smooth structures on $M$ via the cork-twistings $M \mapsto M(f^{n})$.

\vspace{.1in}

Finding infinite order corks is an open problem. Here, we will discuss construction of a weaker version of this object, namely an infinite order {\it loose cork} $(W, f)$, where the  Stein property on $W$ is dropped. We do this by introducing a technique we call a $\delta$-move, which is basically introducing a  $2/3$-cancelling handle pair along an appropriate curve $\delta \subset \partial W$, and then twisting $W$ along the $2$-handle similar to Gluck-twist. Previously, we proposed an infinite cork example in \cite{a5} and then withdrew our paper due to a mistake. More recently, in \cite{g1} R. Gompf has announced a similar example of an infinite order loose cork, with a correct proof. The example discussed here (Figure~\ref{p1}) turns out to coincide with Gompf's (Remark~\ref{disc}), so in particular this $\delta$-move technique explains his proof as well. The difference between our approach and that of \cite{g1} is similar to the difference between  \cite{a6} and  \cite{g2}; namely in our case we use canceling $2/3$-handle pairs to generate $3$-manifold diffeomorphisms, whereas \cite{g1} produces $3$-manifold diffeomorphisms by twisting along their sub-tori. Since our method uses a $3$-handle, initially we do not expect the resulting loose corks obtained this way to have Stein property (at least it is not obvious).  So we still do not know if infinite order corks exist. Let us call a cork $(W,f)$  without the Stein property a {\it loose cork}.

\begin{Thm}
The manifold $W$ in Figure \ref{p1} is an infinite order loose cork (its cork twisting map is a $\delta$-move, which will be described below).
\end{Thm}

  \begin{figure}[ht]  \begin{center}
 \includegraphics[width=.27\textwidth]{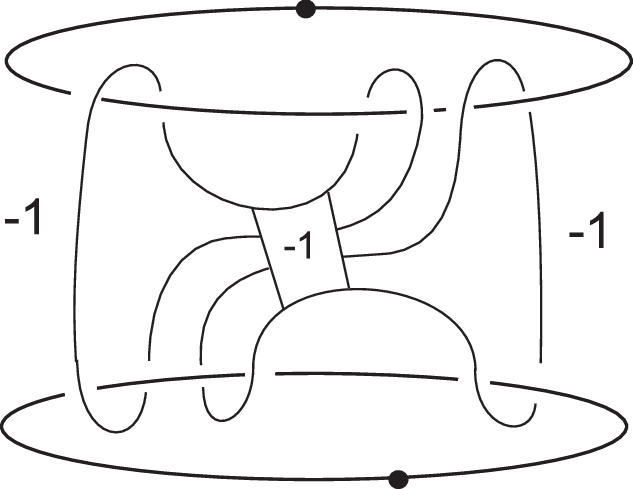}       
\caption{W}      \label{p1} 
\end{center}
 \end{figure}

\vspace{-.1in}

\section{Construction}

First recall the $3$-manifold difeomorphism of Figure~\ref{p2}, that is used to describe the Gluck twisting operation in (e.g. \cite{a4} p.65), which is a $4$-manifold operation. Here we will utilize this operation to generate $3$-manifold diffeomorphisms of the boundary of a $4$-manifold $X$. 

\begin{figure}[ht]  \begin{center}
 \includegraphics[width=.4\textwidth]{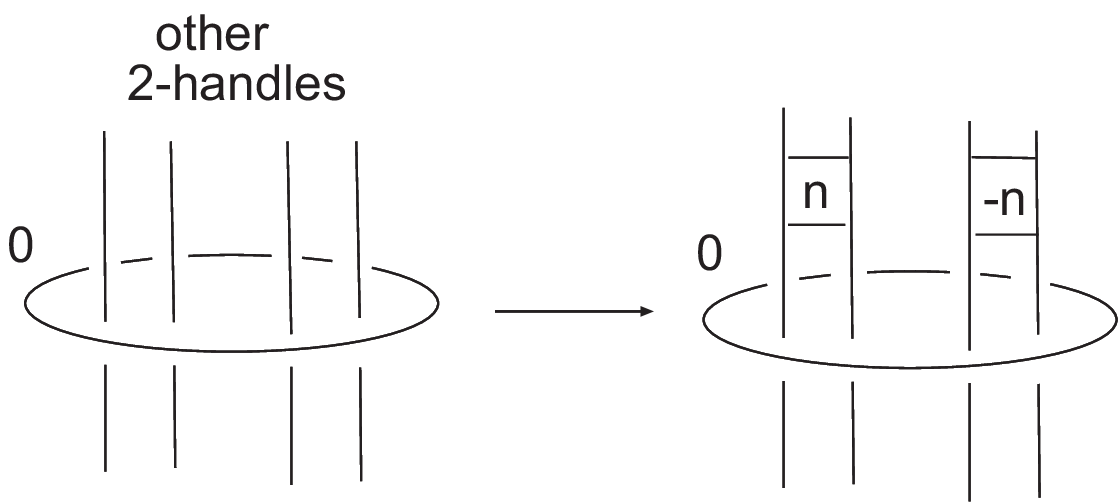}       
\caption{Gluck twisting}      \label{p2} 
\end{center}
 \end{figure}
 
 \newpage
 
 Let $X$ be any $4$-manifold with boundary, and $C$ be a zero framed circle in $\partial X$, and  $\delta \subset X$ be an unknot in $\partial X$, which is obtained by connected summing two parallel copies of  $C$ along some possibly long complicated band (framed arc), i.e. $\delta= C_{+}\cup \mbox{band} \cup C_{-}$ as indicated in Figure~\ref{p3}. 
Define $\delta$-move to be the diffeomorphism 
$$ f_{\delta}:  \partial X\to \partial X$$ 
obtained by first attaching a $2$-handle to $\delta $ with $0$-framing and canceling it with a $3$-handle (recall $\delta$ is an unknot), then blowing up along $C_{+}$ a $1$-framed circle, then sliding it along the $0$-framed $\delta$, and then blowing down along $C_{-}$ circle (as described in Figure~\ref{p3}). At the end of this operation the resulting  right and left twists cancel each other, and we are left with the same picture; so this is a $3$-manifold self-diffeomorphism of $\partial X$. Here we can have any number of framed knots (2-handles) and circles-with-dots ($1$-handles) going through $C$, since this is only a  $3$-manifold diffeomorphism. 
 
  \begin{figure}[ht]  \begin{center}
 \includegraphics[width=.4\textwidth]{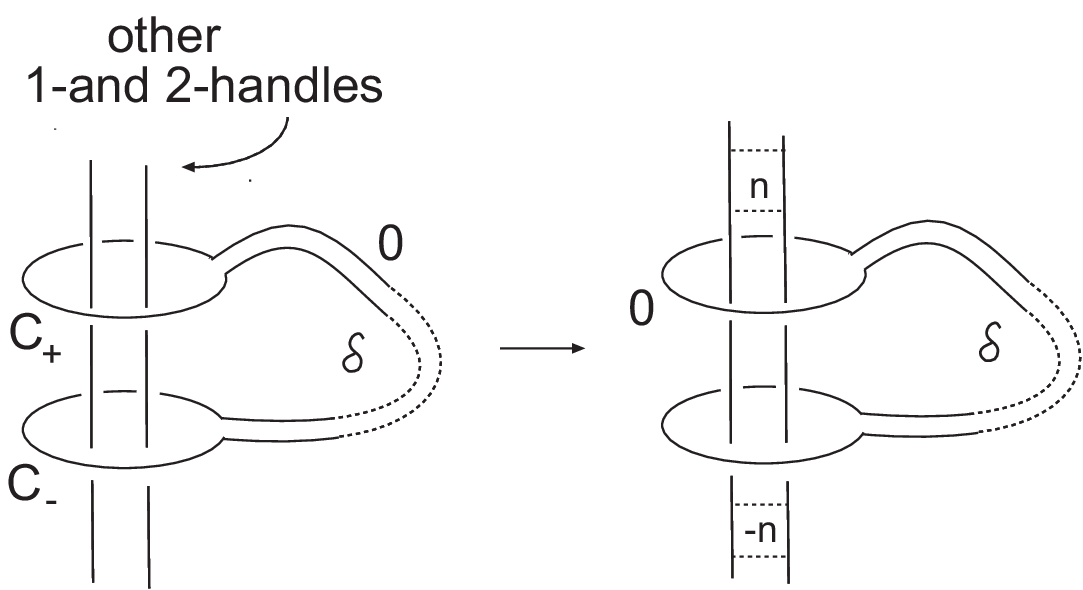}       
\caption{$\delta$--move, performed $n$ times.}      \label{p3} 
\end{center}
 \end{figure}

 To do this operation, we need an interesting $\delta$-loop on the boundary, as described above. For example, in case of the manifold $W$ of Figure~\ref{p1}, we can use the green circle of Figure~\ref{p4} as a $\delta$-loop (check that it is the unknot on the boundary).

 \begin{figure}[ht]  \begin{center}
 \includegraphics[width=.35\textwidth]{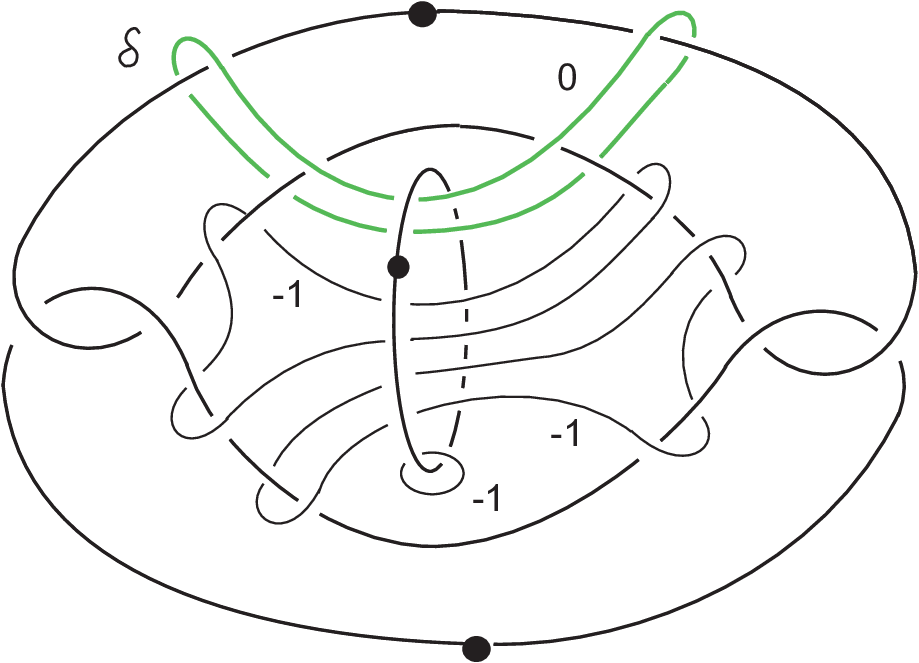}       
\caption{W with a $\delta$-loop}      \label{p4} 
\end{center}
 \end{figure}
 
 \newpage
 
\begin{proof}[Proof of Theorem 1.1] Now construct a new $4$-manifold $S$ from $W$, by attaching three $2$-handles to its boundary along the three framed circles $a, b$ and $c$, with framings $-1, -1$ and $0$, respectively (Figure~\ref{p5}); and then attach a $3$-handle. That is, we extended $W$ to $S$ by gluing a cobordism $H$ to its boundary. Check that the loop $\delta$ is still the unknot in $\partial S$, so we can still use it to introduce $2/3$-handle pair.   Now we can perform $\delta$ moves to $W\subset S$, $n$-times, and obtain the decompositions:
 $$S=W\cup_{id}H \;\;\;\mbox{and}\;\;\; S_{n}=W\cup_{f_{\delta}^{n}}H$$
 
 \begin{figure}[ht]  \begin{center}
 \includegraphics[width=.35\textwidth]{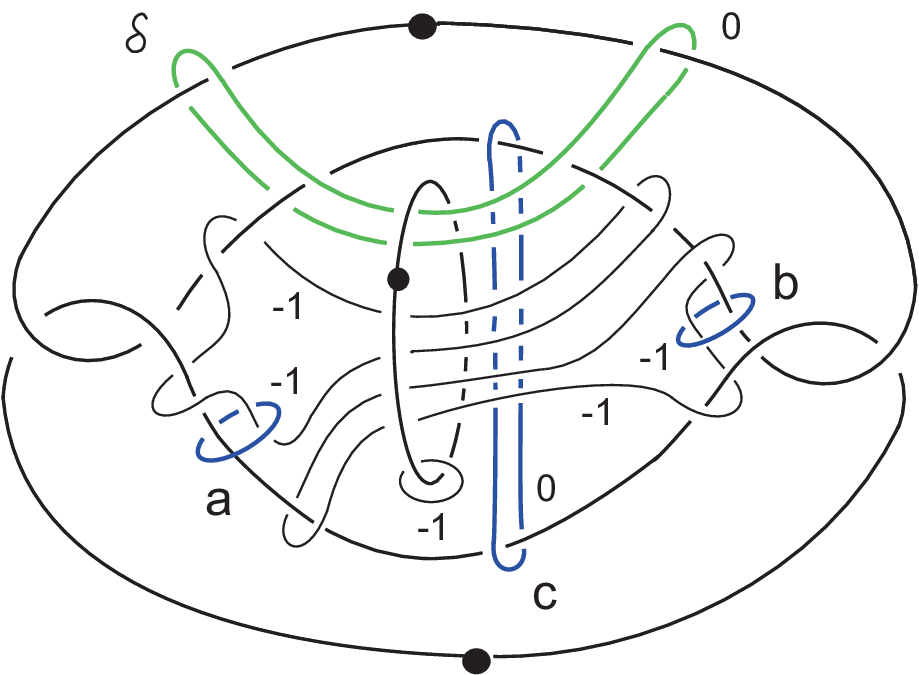}       
\caption{S}      \label{p5} 
\end{center}
 \end{figure}
 
Therefore, by cork twisting $S$ along $W$, by $n$ consecutive 
$\delta$-moves $$f_{\delta}^{n}:\partial W \to \partial W,$$ we get Figure~\ref{p6}, and then by handle slides and canceling (e.g. by performing the operation described in Figure 1.17 of \cite{a4}), we obtain the top picture of Figure~\ref{p7}, which is a description of  $S_{n}$. Since the $0$-framed loop  $\delta$ is the unknot, after the $\delta$-moves we canceled it with a $3$-handle; hence $\delta$ is no longer present in the figures Figures~\ref{p6} and ~\ref{p7}.
 
 \begin{figure}[ht]  \begin{center}
 \includegraphics[width=.35\textwidth]{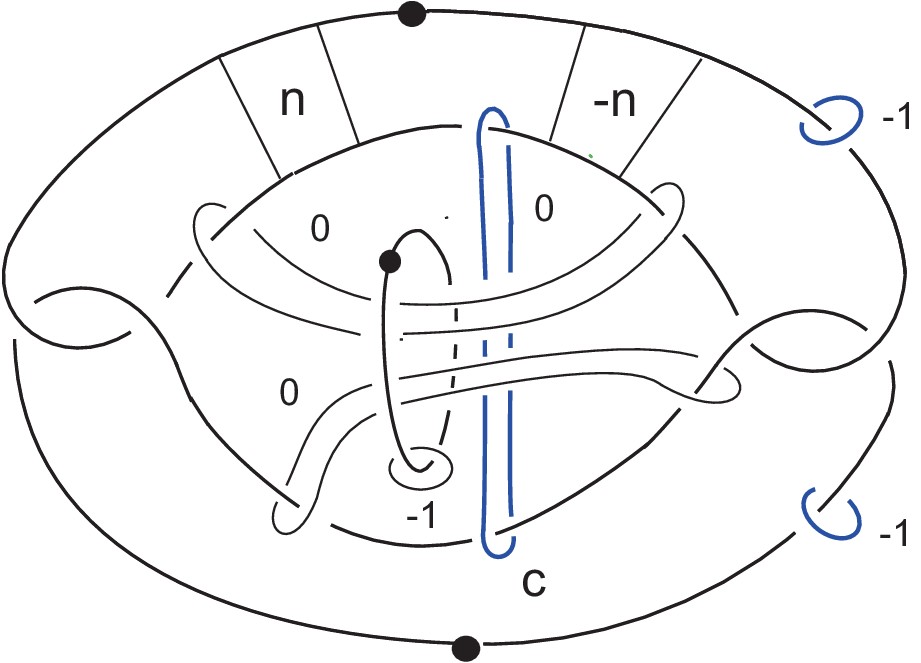}       
\caption{$S_{n}$}      \label{p6} 
\end{center}
 \end{figure}
 
\vspace{.05in}

Now it is easy to check that $S$ is the Stein manifold of Figure~\ref{p8}, and  $S_{n}$ is obtained by the Fintushel-Stern knot surgery operation to S by using the torus inside the cusp $C \subset S$, and using the $n$-twist knot  $K_{n}$ (the algorithm of drawing handlebody of knot surgery operations was introduced and discussed in \cite{a2} and \cite{a3}). By the following Lemma~\ref{ex},  all $S_{n}$ are all mutually nondiffeomorphic exotic copies of $S$, 
and hence $(W, f_{\delta})$ is an infinite order loose cork (it is loose, because we do not know if it is a Stein manifold).

 \begin{figure}[ht]  \begin{center}
 \includegraphics[width=.4\textwidth]{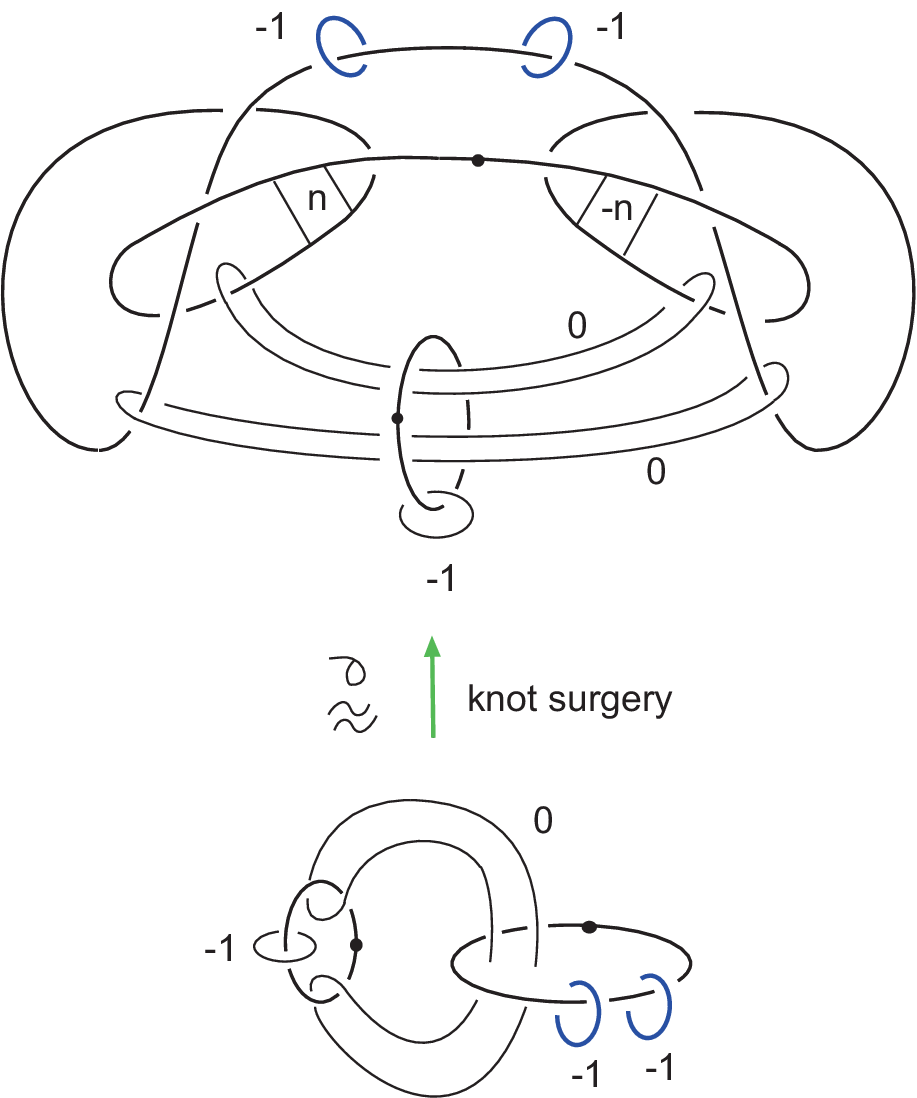}       
\caption{$S_{n}$ is obtained by knot surgery operation 
$S\rightsquigarrow S_{n}$}      \label{p7} 
\end{center}
 \end{figure}
 
 
  \begin{figure}[ht]  \begin{center}
 \includegraphics[width=.2\textwidth]{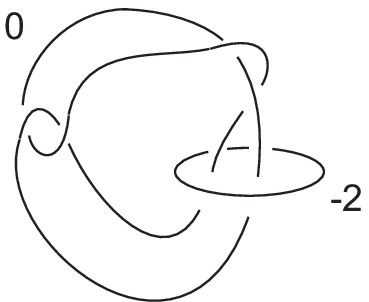}       
\caption{S}      \label{p8} 
\end{center}
 \end{figure}
\end{proof}

\begin{Lem} \label{ex}  For different values of $n=1,2,..$, the manifolds $S_{n}$ are mutually distinct exotic smooth copies of $S$. 
\end{Lem}

\begin{proof} Let $M$  be the closed symplectic manifold with $b_{2}^{+}(M)>1$, obtained by compactifying $S$ (by \cite{lm}, \cite{ao}, or \cite{a4}, p.107). Hence by \cite{t}, $M$ has nontrivial Seiberg--Witten invariant. Let $M_{n}$ be the closed manifold obtained from $M$ by replacing $S$ by $S_{n}$ inside. Hence by \cite{fs}, all $M_{n}$'s have different Seiberg--Witten invariants (because the knots $K_{n}$ have different Alexander polynomials for different values of $n$). So $M_{n}$'s are mutually distinct exotic manifolds. Now we claim that, this implies  $S_{n}$'s are distinct exotic manifolds:  To see this, consider the diffeomorphism $f:\partial S\to \partial S_{n}$ obtained by restricting the identity $M-S \to M_{n}-S_{n}$.  From the construction, we see that
$\partial S$ is a $T^{2}$ bundle over $S^{1}$ with mondromy $A=a^{2}b$, where $a,b$ are the standard Dehn twist generators of $T^{2}$: 

$$a =\small \left(
\begin{array}{cc}
1& -1 \\
0 & 1 \\
\end{array} \right), \;\; 
b =\small \left(
\begin{array}{cc}
1& 0 \\
1 & 1 \\
\end{array} \right) \;\Longrightarrow \;A =\small \left(
\begin{array}{cc}
-1& -2 \\
1 & 1 \\
\end{array} \right)$$

\vspace{.05in}

By standard 3-manifold theory, $f$ can be isotopped to a fiber
preserving isotopy (e.g. \cite{bo}), and let $B\in SL(2,\Z)$ be its action on the fiber. In particular, f has to commute with the monodromy of $\partial S$. Then by solving $AB=BA$, we get $B=\pm I$ or $B=\pm A$. In both cases the corresponding diffeomorphisms, considered as an automorphism,  $f: \partial S \to \partial S$ extends inside $S\to S$ (one is the identity, the other is induced by the PALF structure). So if $S_{n}$ was diffeomorphic to $S$, $f$ would extend to a diffeomorphism $S\to S_{n}$, which implies $M$ would be diffeomporphic to  $M_{n}$, this would be a contradiction. Similarly, to show $S_{n}$ is not diffeomorphic to $S_{m}$ (for $n\neq m$) we prove the diffeomorphism $f$ above extends to a diffeomorphism $S_{n}\mapsto S_{n}$.
\end{proof}

\newpage

\vspace{.05in}

\begin{Rm} \label{disc} In the handlebody picture of $W$ (in Figure~\ref{p1}), we can undo the $-1$ twist across the middle strands, by replacing it with a circle-with-dot with a $-1$ framed circle linking it. We can then cancel the original $-1$ framed $2$-handles of Figure~\ref{p1} with the corresponding $1$-handles, while dragging the newly introduced circle-with-dot along for a ride, which eventually becomes a ribbon knot mentioned in \cite{g1}. \end{Rm}

{Acknowledgements: I would like to thank Bob Gompf for sharing his ideas in \cite{g1}, which motivated me rethink about infinite order corks; also thank Cagri Karakurt for discussing some of the constructions of this paper with me.}

\end{document}